\documentclass{amsart}

\makeatletter
\@namedef{subjclassname@2020}{%
  \textup{2020} Mathematics Subject Classification}
\makeatother

\usepackage{latexsym}
\usepackage{amssymb}
\usepackage{graphicx}
\usepackage{amsthm}
\usepackage{epsfig}
\usepackage{tikz}
\usetikzlibrary{cd}
\RequirePackage{color}

\usepackage{calrsfs}
\DeclareMathAlphabet{\pazocal}{OMS}{zplm}{m}{n}

\usepackage{amscd,enumerate}
\usepackage{amsfonts}

\input xy
\xyoption{all}

\usepackage{multicol}
\usepackage{hyperref}
\hypersetup{ colorlinks,
linkcolor=blue,
filecolor=green,
urlcolor=blue,
citecolor=blue }

\newtheorem{lemma}{Lemma}[section]
\newtheorem{corollary}[lemma]{Corollary}
\newtheorem{theorem}[lemma]{Theorem}
\newtheorem{proposition}[lemma]{Proposition}
\newtheorem{remark}[lemma]{Remark}
\newtheorem{definition}[lemma]{Definition}

\newtheorem{example}[lemma]{Example}

\newtheorem*{theorem*}{Theorem}

\tikzset{commutative diagrams/.cd,
mysymbol/.style={start anchor=center,end anchor=center,draw=none}
}

\def\A{\mathcal{A}}
\newcommand{\alg}{\text{\bf alg}}
\newcommand{\Alg}{\text{\bf Alg}}

\newcommand{\End}{\text{End}}

\definecolor{turquoise2}{rgb}{0,0.898039,0.933333}
\definecolor{magenta}{rgb}{1,0,1}

\def\F{\mathfrak{F}}

\def\Z{\mathbb Z}
\def\Zs{\tiny{\mathbb Z}}
\def\B{\mathcal B} 
\def\rng{\text{\bf rng}}

\def\a{\alpha}

\def\F{\mathcal{F}}
\def\remove#1{}
\begin{document}

\thanks{The  authors are supported by the Spanish Ministerio de Ciencia e Innovaci\'on through the project  PID2019-104236GB-I00 and by the Junta de Andaluc\'{\i}a  through the projects  FQM-336 and UMA18-FEDERJA-119,  all of them with FEDER funds. The fourth author is supported by a Junta de Andalucía PID fellowship no. PREDOC\_00029.}

\title[On isomorphism conditions for algebra functors]{On isomorphism conditions for algebra functors with applications to Leavitt path algebras}

\begin{abstract} 
We introduce certain functors from the category of commutative rings (and related categories) to that of $\Z$-algebras (not necessarily associative or commutative). One of the motivating examples is the Leavitt path algebra functor $R\mapsto L_R(E)$ for a given graph $E$. Our goal is to find  \lq\lq descending\rq\rq\ isomorphism results of the type: if $\F,\mathcal{G}$ are algebra functors and $K\subset K'$ a field extension, under what conditions an isomorphism $\F(K')\cong \mathcal{G}(K')$ of $K'$-algebras implies the existence of an isomorphism $\F(K)\cong\mathcal{G}(K)$ of $K$-algebras? 
We find some positive answers to that problem for the so-called \lq\lq extension invariant functors\rq\rq\ which include the functors associated to Leavitt path algebras, Steinberg algebras, path algebras, group algebras, evolution algebras and others. For our purposes, we employ an extension of the Hilbert's Nullstellensatz Theorem for polynomials in possibly infinitely many variables,
as one of our main tools. We also remark that for extension invariant functors $\F,\mathcal{G}$, an isomorphism
$\F(H)\cong\mathcal{G}(H)$, for some Hopf $K$-algebra $H$,  implies the existence of an isomorphism $\F(S)\cong\mathcal{G}(S)$ for any commutative and unital $K$-algebra $S$.
\end{abstract}

\author[C. Gil]{Crist\'{o}bal Gil Canto}
\address{C. Gil Canto: Departamento de Matem\'atica Aplicada, E.T.S. Ingenier\'\i a Inform\'atica, Universidad de M\'alaga, Campus de Teatinos s/n. 29071 M\'alaga.   Spain.}
\email{cgilc@uma.es}

\author[D. Mart\'{\i}n]{Dolores Mart\'{\i}n Barquero}
\address{D. Mart\'{\i}n Barquero: Departamento de Matem\'atica Aplicada, Escuela de Ingenier\'{\i}as Industriales, Universidad de M\'alaga. 29071 M\'alaga. Spain.}
\email{dmartin@uma.es}

\author[C. Mart\'{\i}n]{C\'andido Mart\'{\i}n Gonz\'alez}
\address{C. Mart\'{\i}n Gonz\'alez:  Departamento de \'Algebra Geometr\'{\i}a y Topolog\'{\i}a, Fa\-cultad de Ciencias, Universidad de M\'alaga, Campus de Teatinos s/n. 29071 M\'alaga. Spain.}
\email{candido\_m@uma.es}
\author[I. Ruiz]{Iv\'an Ruiz Campos}
\address{I. Ruiz Campos:  Departamento de \'Algebra Geometr\'{\i}a y Topolog\'{\i}a, Fa\-cultad de Ciencias, Universidad de M\'alaga, Campus de Teatinos s/n. 29071 M\'alaga. Spain.}
\email{ivaruicam@uma.es}

\maketitle

\section{Introduction}

One of the isomorphism problems for Leavitt path algebras is the following: when $E_1$ and $E_2$ are graphs, $K$ a field and $L_K(E_1)\cong L_K(E_2)$ is an isomorphism of $K$-algebras, not necessarily preserving the diagonal and not necessarily graded, can we have a $K'$-algebra isomorphism $L_{K'}(E_1)\cong L_{K'}(E_2)$ for some other field $K'$? In the literature, there are positive answers to closely related problems. For example, in \cite{BCH} it is proved that the graph groupoids of directed graphs $E_1$ and $E_2$ are topologically isomorphic if and only if there is a diagonal-preserving ring $*$-isomorphisms between the Leavitt path algebras of $E_1$ and $E_2$. More generally, we can find questions dealing with the equivalence between the existence of graded isomorphism of groupoids and diagonal-preserving graded isomorphism of Steinberg algebras, like in the works \cite{ABHS}, \cite{CR} and \cite{S}. Another isomorphism-type problem is the following: assume $K\subset K'$ is a field extension and $L_{K'}(E_1)\cong L_{K'}(E_2)$ an isomorphism of $K'$-algebras (not necessarily preserving the diagonal). We would like to know under what conditions the above isomorphism implies the existence of an isomorphism $L_K(E_1)\cong L_K(E_2)$. This kind of results could be termed \lq\lq descending conditions for isomorphism\rq\rq. In case such a descending isomorphism condition exists for an extension $K\subset K'$, then $L_{K'}(E_1)\cong L_{K'}(E_2)$ implies $L_{K''}(E_1)\cong L_{K''}(E_2)$ for any superfield $K''\supset K$.
When trying to give partial solutions to these questions, we find that there is no need to focus on Leavitt path algebras. We can answer them using suitable functors which the Leavitt path algebra functor $R\mapsto L_R(E)$ is a particular case of. Such are the so-called extension invariant functors. One of the main tools that we use is the Hilbert's Nullstellensatz Theorem for polynomials with possibly infinitely many variables. This is a result of S. Lang \cite{L52}
which generalizes the classical Nullstellensatz Theorem but imposes an additional condition on the cardinal of the ground field.

The paper is organized as follows: in Section \ref{preliminares} we first introduce the concept of algebra functor and, in Definition \ref{defextension}, we establish what is the $K$-algebra functor $\underline{A}$ for a certain $K$-algebra $A$. On the one hand, in Definition \ref{vamonospaRonda}, we introduce the extension invariant functors, which will generalize the \lq\lq nice\rq\rq\  condition that the Leavitt path algebras satisfy. In fact, we will see that the extension invariant functors are precisely those of the form $\underline{A}$ for a certain $K$-algebra $A$. On the other hand, in Corollary \ref{Leoquierealgo} we remark that, for extension invariant functors $\F,\mathcal{G}$, once we have an isomorphism $\F(H)\cong\mathcal{G}(H)$ for some Hopf $K$-algebra $H$, this induces an isomorphism $\F(K)\cong\mathcal{G}(K)$. Hence $\F(S)\cong\mathcal{G}(S)$ for 
every commutative and unital $K$-algebra $S$. Besides, we will illustrate, in Example \ref{ejemploderivaciones}, an algebra functor which is not extension invariant thanks to derivations of an algebra. In Subsection \ref{subseccionpathalgebra}, we include some well-known interesting examples of algebras which provide extension invariant algebra functors. It is worth to mention that they are the main motivation for the results presented in this work. As a first corollary we get that for any two graphs $E_1,E_2$ and any Hopf $K$-algebra $H$, then $L_H(E_1)\cong L_H(E_2)$ if and only if $L_K(E_1)\cong L_K(E_2)$ (see Corollary \ref{Leavitthopf}).

Next, in Section \ref{seccionNull}, we will present the key to develop our machinery for the isomorphism theorems. According to \cite{L52}, the Hilbert's Nullstellensatz Theorem can be given in terms of infinite indeterminates. To achieve the hypotheses for Theorem \ref{extension}, it is necessary that either the set of indeterminates is finite (in fact, the original Hilbert's Nullstellensatz Theorem) or that the transcendence degree of the ground field over its prime one has cardinality strictly higher than the number of variables. 
To unify those hypotheses in a single condition, we introduce Lemmas \ref{Leo} and \ref{Leo2}. This new environment consists of that the cardinal of the ground field is strictly higher that the cardinal of the set of indeterminates. The original paper \cite{L52} by S. Lang informs that I. Kaplansky obtained independently the result, but the equivalent formulation of Kaplansky was slightly different and did not use transcendence degree. Since Kaplansky never published his work, this encourages us to include the proofs of Lemmas \ref{Leo} and \ref{Leo2}.

The remaining two sections are devoted, one for the finite dimensional issue (Section \ref{fdc}) and the other for the infinite one (Section \ref{idc}). The first case gives us a better understanding of the general one. The idea in both situations is analogous: we need to translate respectively the isomorphism problem into the existence of zeros of a certain ideal of polynomials. This corresponds to Proposition \ref{tecnicoUNO} and Corollary \ref{correspondencia}, for the finite dimensional case, and Proposition \ref{correspondenciaInf}, for the general one. For the finite dimension, as a consequence of Theorem \ref{extension} and Corollary \ref{correspondencia}, we prove Theorem \ref{VueltaAlCole} where for a field extension $K \subset F$ with $K$ algebraically closed, the isomorphism between $A \otimes_K F$ and $B \otimes_K F$ (as $F$-algebras) is equivalent to the isomorphism between $A$ and $B$ (as $K$-algebras). From this main assertion, varying the hypotheses on the field extensions, it will be derived a series of Corollaries \ref{mismaCaracteristica}, \ref{dosalgcerrados}, \ref{primefield}.  We complete the finite dimensional case by establishing the previous assertions in terms of extension invariant algebra functors giving a more global vision of the problem; these are Theorems \ref{zapaterobis} and \ref{zapatero}. In particular, in order to highlight our motivation, we assert Corollary \ref{granuja} and Corollary \ref{granujilla} in terms of relative Cohn path algebras.

The same idea is pursued in the infinite-dimensional section by Theorem \ref{VueltaAlCole2}. However, we need the additional hypothesis that the cardinal of the ground field $K$ is strictly higher that the dimension of the algebras. This hypothesis is essential in order to apply the Hilbert's Nullstellensatz Theorem for infinitely many variables. Along the lines of the finite dimensional case, Corollaries \ref{mismaCaracteristica2}, \ref{dosalgcerrados2} and \ref{primefield2} are followed jointly with a functorial interpretation for those extension invariant ones (Theorem \ref{VueltaAlCole2Funtorial} and Theorem \ref{aporotracosa}). Due to the fact that Leavitt path algebras become our major motivation, we include results that could be stated in the general setting of extension invariant functors, but that we expose in our particular favorite ambient. So, for instance, for countable graphs $E_1$ and $E_2$, an uncountable algebraically closed field $K$ and $K \subseteq F_1, F_2$, then $L_{F_1}(E_1)\cong L_{F_1}(E_2)$ if and only if $L_{F_2}(E_1)\cong L_{F_2}(E_2)$ (Corollary \ref{andale}).

\section{Preliminaries and first results}\label{preliminares}

For a commutative, associative and unital ring $K$ denote by $\alg_K$ the category of associative and commutative $K$-algebras with non-zero unit. We will also use the category $\Alg_K$ of all $K$-algebras (not necessarily associative, commutative or unital). A functor $\F\colon\alg_K\to\Alg_K$ such that $\F(R)$ is an $R$-algebra for any $R\in\alg_K$ will be called a {\em $K$-algebra functor.} 
We will identify the category $\alg_{\Zs}$ with the category of (associative commutative and unital) rings, usually denoted $\rng$.
Some well-known examples are in order. Take a positive integer $n$ and consider the functor $M_n\colon\rng\to\Alg_{\Zs}$ such that $M_n(R)$ is the $R$-algebra of $n\times n$ matrices with coefficients in $R$ with the usual matrix product. This functor produces associative algebras but this is not true for the functor $\text{sl}_n\colon\rng\to\Alg_{\Zs}$ such that $\text{sl}_n(R)$ is defined as the set of all zero-trace matrices of $M_n(R)$ with the Lie product $[x,y]:=xy-yx$ for any $x,y$. Another source of algebra functors, that we will use in the sequel, is the following: 

\begin{definition}\label{defextension}\rm
Let $A \in \Alg_K$, we are able to define the $K$-algebra functor  $\underline{A} \colon \alg_K \to \Alg_K$ given by $\underline{A}(R):=A \otimes_{K} R = A_R$. Furthermore, for any $R,R' \in \alg_K$, if $\alpha: R \to R'$ is a morphism in $\alg_K$, then $\underline{A}(\alpha) \colon \underline{A}(R) \to \underline{A}(R')$ is defined as $\underline{A}(\alpha)=1\otimes_K \alpha \colon A \otimes_K R \to A \otimes_K R'$ such that $a\otimes_K r\mapsto a\otimes_K \alpha(r)$.
\end{definition}

When it is understood by the context, the tensor product will be simply denoted by $\otimes$ instead of $\otimes_K$.

As we previously mentioned, this set of functors are well behaved. If $K$ is a commutative and unital ring, $R \in\alg_K$, $S\in\alg_R$ and $A \in\Alg_K$, then it is straightforward to prove that $(A \otimes_K R) \otimes_R S \cong A \otimes_K (R \otimes_R S)$. As a consequence, $(A_R)_S\cong A_S$ as $S$-algebras. For more details, define $\omega_{RS}\colon (A_R)_S\to A_S$ the isomorphism (of $S$-algebras) such that $\omega_{RS}((a\otimes_K r)\otimes_R s):=a\otimes_R rs$. This isomorphism is natural in $S$, that is, for any $\beta\in\hom_{\alg_R}(S,S')$ we have the commutativity of the diagram:
\begin{equation}\label{cuadradoomega}
\begin{tikzcd}
(A_R)_S \arrow[r, "\omega_{RS}"] \arrow[d, "1_{A_R}\otimes_R \beta"']
& A_S \arrow[d, "1_{A}\otimes_K \beta "] \\
(A_R)_{S'} \arrow[r,  "\omega_{RS'}"']
& A_{S'}
\end{tikzcd}
\end{equation}

One important property associated to $K$-algebra functors that we will use along this work, is introduced now.

\begin{definition}\label{vamonospaRonda}\rm
Let $K$ be a commutative and unital ring. A $K$-algebra functor $\F\colon\alg_K\to\Alg_K$ is said to be \emph{extension invariant} if 
for every $R\in\alg_K$ there is an $R$-algebra isomorphism $\tau_R\colon \F(R)\to\F(K)\otimes_K R$ in $\Alg_K$ which is natural in $R$, that is,
for any $\alpha \colon R \to R'$ homomorphism between $R,R' \in \alg_K$ the following diagram is commutative:
\begin{equation}\label{cuadradotao}
\begin{tikzcd}[cramped, sep=large]
\F(R) \arrow[r, "\tau_R"] \arrow[d, "\F{}(\a)"']
& \F(K)\otimes_{K} R \arrow[d, "1_{\F(K)}\otimes_K\alpha"] \\
\F(R') \arrow[r,  "\tau_{R'}"']
& \F(K)\otimes_K R'
\end{tikzcd}
\end{equation}
\end{definition}

\begin{proposition}\label{extenprop}
If $\F$ is an extension invariant functor, then
$\F(S)\cong \F(R)\otimes_R S$ as $S$-algebras for any $R\in\alg_K$ and $S\in\alg_R$. Furthermore, the above isomorphism can be chosen to be natural in $S$.
\end{proposition}
\begin{proof} The isomorphism is the composition $\tau_S^{-1}\omega_{RS}(\tau_R\otimes 1_S)$ and it is natural in $S$ because for any homomorphism of $R$-algebras $\beta \colon S\to S'$, we have the commutativity of the three squares in the diagram:
\begin{equation*}
\begin{tikzcd}
\F(R)\otimes_R S \arrow[r, "\tau_R \otimes 1_S"] \arrow[d, "1_{\F(R)}\otimes_R \beta"'] &(\F(K)\otimes_K R)\otimes_R S \arrow[r, "\omega_{RS}"]\arrow[d, "1_{\F(K)\otimes R}\otimes_R \beta"'] &\F(K)\otimes_K S\arrow[r, "\tau_S^{-1}"]\arrow[d, "1_{\F(K)}\otimes_K \beta"'] & \F(S)\arrow[d, "\F(\beta)"]
\\
\F(R)\otimes_R S' \arrow[r, "\tau_R \otimes 1_{S'}"]  &(\F(K)\otimes_K R)\otimes_R S' \arrow[r, "\omega_{RS'}"] &\F(K)\otimes_K S'\arrow[r, "\tau_{S'}^{-1}"] & \F(S')
\end{tikzcd}
\end{equation*}
because of the commutativity of \eqref{cuadradoomega} and \eqref{cuadradotao}.
\end{proof}

\begin{remark}
\label{nuevomac}\rm
Observe that if $\F \colon \alg_K \to \Alg_K$ is a $K$-algebra functor which is extension invariant, we can consider $\F(K)$ as a $K$-algebra, then $\F \cong \underline{\F(K)}$. Conversely, if  $A \in \Alg_K$, then the functor  $\underline{A}$ is extension invariant. That is, the extension invariant functors are precisely those algebra functors which are (naturally) isomorphic to some $\underline{A}$. 
\end{remark}

\begin{corollary}\label{Leoquierealgo}
If $\F,\mathcal{G} \colon\alg_K\to\Alg_K$ are extension invariant functors and $\varphi\colon R\to S$ a homomorphism in $\alg_K$, then 
$\F(R)\cong\mathcal{G}(R)$ implies $\F(S)\cong\mathcal{G}(S)$. In particular for $R \in \alg_K$, if we have an homomorphism $\psi \colon R \to K$, then $\F(R)\cong\mathcal{G}(R)$ implies that $\F(K)\cong\mathcal{G}(K)$, which means that $\F(S)\cong\mathcal{G}(S)$ for every $S \in \alg_K$. 
\end{corollary}

Note that Corollary \ref{Leoquierealgo} is true for any \emph{Hopf algebra} since by definition a Hopf algebra is endowed with a homomorphism to the ground field (see \cite[Section 1.4 (page 8)]{Waterhouse}). In other words, we can verify that two extension invariant functors are isomorphic just by checking in a Hopf algebra (for instance in a group algebra). It will be considered below in subsection \ref{subseccionpathalgebra}.

Next, we recall an extension algebra functor from which we will derive some  subfunctors (derivations, centroid).

\begin{example}
\rm Let now $V$ be a vector space over a field $K$ and denote by
$\End_K(V)$ the $K$-algebra of all linear maps $V\to V$. 
If $R\in\alg_K$, we will also denote $\End_R(V_R)$ the $R$-algebra of all $R$-module endomorphisms $V_R\to V_R$. 
There is a functor $\mathfrak{G}\colon\alg_K\to\Alg_K$ such that
$\mathfrak{G}(R):=\End_R(V_R)$ and for any $\a\colon R\to R'$ (in $\alg_K$) we have $\mathfrak{G}(\a)\colon\End_R(V_R)\to\End_{R'}(V_{R'})$ given by 
$\mathfrak{G}(\a)(T)=T'$ such that $T'\colon V_{R'}\to V_{R'}$ 
is defined by $T'(v\otimes 1):=(1\otimes\a)T(v\otimes 1)$
whenever $T\in\End_R(V_R)$.  We will call $\mathfrak{G}$ the
{\em endomorphism functor}. This algebra functor is well-known and it is easy to prove that is extension invariant thanks to the natural isomorphism $\Omega \colon \End(V) \otimes R \to \End(V_R)$, such that $\Omega(T\otimes r)(v \otimes s)=T(v) \otimes rs$.
\end{example}

\begin{example}\rm
For an algebra $A$ over a unital commutative associative ring $K$, the \emph{centroid} of $A$ is the space of $K$-linear transformations $T$ on $A$ such that $T(ab)=aT(b)=T(a)b$ for all $a,b \in A$ and we will denote it by $C(A)$. Besides, the \emph{center} of $A$ is the set $Z(A):=\{a \in A \colon ax=xa \text{ for all } x \in A\}$. If $A$ is an associative, unital $K$-algebra we have that $Z(A)\cong C(A)$. So, the subfunctor $\F\colon\alg_K\to\Alg_K$ 
of the endomorphism functor $\mathfrak{G}$
such that $\F(R):=C(A_R)$ is extension invariant. If we consider the $K$-algebra $A=M_{\infty}(K)$, where $M_\infty(K)$ denotes the set of infinite square matrices  with a finite number of nonzero entries over $K$, though $A$ is not unital, the functor $\F$ is still extension invariant. Indeed, it is well known that $C(A_R)=R\cdot 1_{A_R}$ and so the isomorphism 
$\tau_R\colon C(A_R)\to C(A_K)\otimes R$ is just 
$\tau_R(r 1_{A_R})=1_{A}\otimes r$ which is natural in $R$.
\end{example}

The following example gives an algebra functor which is not extension invariant.

\begin{example} \label{ejemploderivaciones}\rm Denote by $\text{Der}$ the \emph{derivations} of the corresponding algebra, that is to say, all the linear applications $d$ satisfying $d(ab)=a d(b)+d(a) b$ for any elements $a,b$ in the algebra. Let us consider a perfect finite-dimensional nontrivial $K$-algebra $U$, 
it can be proved that $R\mapsto\text{Der}(U_R)$ is the object map of a functor $\F \colon \alg_K \to \Alg_K$ which is a subfunctor 
of the endomorphism functor $\mathfrak{G}$. 
This functor is not extension invariant. Observe that $\F(K)=\text{Der}(U_K) \cong \text{Der}(U)$. Following \cite[Theorem 1]{BM} or \cite[Remark 2.31]{BN} we have that $\text{Der}(U \otimes K[x]) \cong \text{Der}(U) \otimes K[x] \oplus C(U) \otimes \text{Der}(K[x])$. Taking into account that $\text{Der}(K[x]) \neq 0$ and $C(U) \neq 0$ we have that $\F(K[x]) \ncong \F(K)\otimes_K K[x]$. In fact, this construction gives us an example of a non extension invariant subfunctor of an extension invariant functor. 
\end{example}

\subsection{Path, Cohn path, Leavitt path, Steinberg and evolution algebra functors and others}\label{subseccionpathalgebra}

 In this subsection, we will explore some interesting examples of extension invariant functors which, in fact, are not subfunctors of endomorphims. We are talking about the path, Leavitt path, Cohn path and Steinberg algebra functors among others. 

First, a bunch of algebra functors is provided by the many different graph algebras which one can associate to a directed graph $E$. A \emph{directed graph} $E=(E^0,E^1,r,s)$ consists of two sets $E^0$ and $E^1$  and functions $r,s\colon E^1 \to E^0$. The elements of $E^0$ are called \emph{vertices} and the elements of $E^1$ \emph{edges}. We will mention that a graph $E$ is finite (respectively countable)  if $\vert E^0 \cup E^1 \vert < \infty$ where $\vert \cdot \vert$ denotes the cardinal of the corresponding set (respectively $\vert E^0 \cup E^1 \vert \leq  \aleph_0$). 
A vertex $v$ is called a \emph{regular vertex} if $s^{-1}(v)$ is a finite non-empty set. The set of regular vertices is denoted by ${\rm Reg}(E)$. A \emph{path of length $n$} in $E$ is a sequence $e_1 e_2 \ldots e_n$ of edges in $E$ such that $r(e_i)=s(e_{i+1})$ for $i \in \{1,2,\ldots,n-1\}$. If $\xi$ is a path of length $n$, then we write $|\xi|=n$.  We consider vertices in $E^0$ as paths of length zero. The set of all finite paths of length $n$ is denoted by $E^n$ and we let ${\rm Path}(E)=\cup_{n=0}^{\infty} E^n$. For instance, 
for $R \in \rng$, consider the \emph{path algebra} $RE$ as the free $R$-algebra generated by 
the sets $E^0 \cup E^1$ with relations:
\begin{itemize}
    \item[(V)] $vw=\delta_{v,w}v$ for $v,w \in E^0$,
    \item[(E1)] $s(e)e=er(e)=e$ for $e \in E^1$,
\end{itemize}
that is, $RE=\{ \sum_{i\in I} r_i \lambda_i : r_i \in R, \lambda_i \in \text{Path}(E) \hbox{ and }\vert I\vert \hbox{ finite}\}$. 
This suggests the definition of the functor $PE \colon \rng \to \Alg_{\Z}$ given by $R \mapsto RE$ and if $\alpha \colon R \to R'$ is a morphism in $\rng$, then $PE(\alpha)\left (\sum_{i \in I}r_i\lambda_{i} \right )=\sum_{i\in I} \alpha (r_i) \lambda_{i}$. It is well known that $PE$ is an extension invariant functor.

Denote by $X_E$  any subset of ${\rm Reg}(E)$ or by $X$ if  there is no confusion. The \emph{Cohn path algebra of $E$ relative to $X_E$}, denoted $C_R^{X_E}(E)$, is the free $R$-algebra generated by the sets $E^0 \cup E^1 \cup \{e^* \ | \ e \in E^1\}$ with relations (V), (E1) and:
\begin{itemize}
    \item[(E2)] $r(e)e^*=e^*s(e)=e^*$ for $e \in E^1$,
    \item[(CK1)] $e^*f = \delta_{e,f}r(e)$ for $e,f \in E^1$ and,
    \item[(XCK2)] $v = \sum_{e \in s^{-1}(v)}ee^*$ for every vertex $v \in X$.
\end{itemize}
As a consequence, we have that the \emph{Cohn path algebra} $C_R(E)$ corresponds to $C_R(E)=C_R^{\emptyset}(E)$ and the \emph{Leavitt path algebra} $L_R(E)$ to $L_R(E)=C_R^{{\rm Reg}(E)}(E)$. We have that every element of $C_R^X(E)$ can be represented as a sum of the form $\sum_{i=1}^n  r_i \alpha_i\beta_i^*$ for some $n \in \mathbb{N}$, paths $\alpha_i, \beta_i$ such that $r(\alpha_i)=r(\beta_i)$, and $r_i \in R$ for every $i=1, \ldots,n$. For more details about these graph algebras see \cite[Chapter 1]{AAS}. These algebras provide functors
$C^X(E)\colon\rng\to\Alg_\Z$ such that $R\mapsto C_R^X(E)$. In particular, we have $C^{\rm{Reg(E)}}(E)=L(E)\colon\rng\to\Alg_\Z$ such that $R\mapsto L_R(E)$. As an important fact, by \cite[Corollary 1.5.14]{AAS}, we have that $C_S^{X}(E) \cong C_R^{X}(E) \otimes_R S$ for any unital commutative ring $R$ and any unital commutative $R$-algebra $S$. It is easy to check that this isomorphism is natural, therefore the functor $C^X(E)$ is extension invariant.

At the same time, we may give another example of functors provided by Steinberg algebras. Let $G$ be a groupoid, that is, $G$ generalizes the concept of group being the binary operation only partially defined. The unit space of $G$, denoted by $G^{(0)}$, is $G^{(0)}=\{\gamma \gamma^{-1} \colon \gamma \in G\}=\{\gamma^{-1} \gamma \colon \gamma \in G\}$.
Groupoid source and range maps ${\mathfrak r},\mathfrak{s} \colon G \to G^{(0)}$ are defined by ${\mathfrak s}(\gamma)=\gamma^{-1} \gamma$ and ${\mathfrak r}(\gamma)=\gamma \gamma^{-1}$. Steinberg algebras are algebras associated to Hausdorff ample groupoids. We briefly remind their construction. First, a {\it topological groupoid} is a groupoid equipped with a topology such that composition and inversion are continuous. An open subset $B$ of a groupoid $G$ is an {\it open bisection} if ${\mathfrak r}$ and ${\mathfrak s}$ restricted to $B$ are homeomorphisms onto an open subset of $G^{(0)}$. A topological groupoid is said to be \emph{ample} if there is a basis for its topology consisting of compact open bisections. So at the end, suppose that $G$ is a Hausdorff ample groupoid and $R \in \rng$, the \emph{Steinberg algebra} associated to $G$, denoted $A_R(G)$, is the $R$-algebra of all locally constant functions $f \colon G \to R$ such that ${\rm supp} f=\{\gamma \in G \colon f(\gamma) \ne 0\}$ is compact. For $f \in A_R(G)$, we have that ${\rm supp} f$ is both compact and open. Thus we have
$$A_R(G)={\rm span}\{ 1_B \colon B \text{ is a compact open bisection}\}$$
where $1_B \colon G \to R$ is the characteristic function of $B$. Addition and scalar multiplication in $A_R(G)$ are defined pointwise and multiplication is given by convolution such that for compact open bisections $B$ and $D$ we have $1_B1_D=1_{BD}$. These algebras provide functors
$\mathfrak{A}(G)\colon\rng\to\Alg_\Z$ such that $R\mapsto A_R(
G)$. And for an homomorphism $\alpha \colon  R \to R'$ we define $\mathfrak{A}(G)(\alpha) \colon A_R(G) \to A_{R'}(G)$ as $\mathfrak{A}(G)(\alpha)\left(\sum_i r_i1_{B_i}\right) = \sum_i \alpha(r_i)1_{B_i}$. It can be easily proved that this is an extension invariant functor.

Let ${\mathcal E}$ be an \emph{evolution $K$-algebra} for a ring $K$, this is nothing but a $K$-algebra which is free as a $K$-module and has a basis $\{v_i\}_{i\in I}$ such that 
$v_iv_j=0$ when $i\ne j$.
Then, if we define as before the functor $\underline{\mathcal E}\colon\alg_{K}\to\Alg_{K}$ such that $\underline{\mathcal E}(R)={\mathcal E}_R={\mathcal E}\otimes_{K} R$, it turns out that $\underline{\mathcal E}(R)$ is an evolution algebra. In case $K=\Z$ we can consider the functor $\underline{\mathcal E}\colon\rng\to\Alg_\Z$ so that $\underline{\mathcal E}(R)={\mathcal E}\otimes_{\Z} R$. In particular for the functor $\underline{\mathcal E}\colon\alg_{K}\to\Alg_{K}$, it can be applied Remark \ref{nuevomac}.

Finally, if $G$ is a group and $K$ a ring we may consider the group ring functor $\underline{KG}\colon\alg_K\to\Alg_K$ such that 
$R\mapsto RG$. For any $\a\in\hom_{\alg_K}(R,R')$, we define 
$\underline{KG}(\a)\left( \sum_i r_i g_i\right)= \sum_i\a(r_i)g_i$
for $r_i\in R$ and $g_i\in G$. Again, Remark \ref{nuevomac} applies. 

Since any group algebra is a Hopf algebra, observe that we are in the conditions of Corollary \ref{Leoquierealgo}: we deduce that if for some group $G$ and extension invariant functors $\F$ and $\mathcal{G}$ we have $\F(KG)\cong\mathcal{G}(KG)$, then $\F(S)\cong\mathcal{G}(S)$ for any $S\in\alg_K$. 

Also, as a consequence of Corollary \ref{Leoquierealgo} and due to the fact that the Leavitt path algebra functor is extension invariant, we conclude our first isomorphism result in the following context:

\begin{corollary}\label{Leavitthopf} If $E_1$ and $E_2$ are graphs and $H$ any Hopf $K$-algebra (in particular any group $K$-algebra), then $L_H(E_1)\cong L_H(E_2)$ if and only if $L_K(E_1)\cong L_K(E_2)$.
\end{corollary}

For instance, if $H$ is a polynomial $K$-algebra $K[x_1,\ldots,x_n]$ or a Laurent polynomial $K$-algebra $K[x_1^{\pm},\ldots,x_n^{\pm}]$,  both of which are Hopf algebras, $L_H(E_1)\cong L_H(E_2)$ implies
$L_K(E_1)\cong L_K(E_2)$.

\section{On the Hilbert's Nullstellensatz Theorem}\label{seccionNull}

Let $\A$ be an indexing set. Consider $K[x_\a]$ the polynomial ring in the indeterminates $\{x_\a\}_{\a\in\mathcal{A}}$ with coefficients in the field $K$ and an ideal $\mathfrak{i} \triangleleft K[x_\a]$. Following \cite{L52}, a \emph{zero} of $\mathfrak{i}$ is an element $(\xi_\a)_{\alpha \in \A}$ where each $\xi_\a$ is in some extension field of $K$ such that $p(\xi_\a)=0$ for all $p\in\mathfrak{i}$. A zero of $\mathfrak{i}$ will be called \emph{algebraic} if all $\xi_\a$ lie in $K$. The set of all algebraic zeros of an
ideal $\mathfrak{i}$ will be called the \emph{variety defined by} $\mathfrak{i}$ (usually denoted $V(\mathfrak{i})$), concretely
$$V(\mathfrak{i})=\{a\in K^{\mathcal{A}}\colon p(a)=0 \text{ for any } p\in \mathfrak{i}\}.$$

The Hilbert's Nullstellensatz Theorem is in general not valid if the number of indeterminates is infinite. However, the main result of \cite{L52}, that we recall in this section for self-containedness, holds. 
\begin{theorem}\label{extension} \cite{L52} Let $K$ be an algebraically closed field and $\{x_\a\}$ a set of indeterminates indexed by $\A$. The following three statements are equivalent:
\begin{enumerate}
\item[S1.] If $\mathfrak{i}$ is an ideal of $K[x_\a]$ and $p\in K[x_\a]$ vanishes on the variety defined by $\mathfrak{i}$, then $p^m\in\mathfrak{i}$ for some integer $m$.

\item[S2.] If $\mathfrak{i}$ is an ideal of $K[x_\a]$ and $\mathfrak{i}\ne K[x_\a]$, then $\mathfrak{i}$ has an algebraic zero.

\item[S3.] A ring extension $K[\xi_\a]$ by elements $\xi_\a$ in some extension field of $K$ is a field if and only if all $\xi_\a$ lie in $K$.
\end{enumerate}

Furthermore, the three statements hold if  one of the following two conditions is satisfied:
\begin{enumerate}
\item[(i)] $\A$ is a finite set.
\item[(ii)]\label{itemtranscendente} Let $\A$ have cardinality $a$. Let the transcendence degree of $K$ over the prime field $Q$ have cardinality $b$. Then $a <b$.
\end{enumerate}
\end{theorem}

As pointed out by S. Lang in his paper \cite{L52},  Professor I. Kaplansky  informed him that he had obtained the above theorem independently.
It is also mentioned in \cite{L52} that the two conditions {\it(i)} and {\it(ii)} can be replaced by the single condition that the cardinality of the field $K$ itself should be greater than the cardinality of $\A$. We include a proof for self-containedness.
\begin{lemma}\label{Leo}
If $Q$ is a finite field or the field of rationals and $\{x_{\a}\}_{\a \in \A}$ an infinite set of variables, then the extension field $Q(x_{\a})$ of $Q$ containing all the indeterminates $x_{\a}$ has cardinal $\vert \A \vert$. As a consequence, if $\{b_{i}\}_{i \in I}$ is an infinite transcendence basis of an extension field $K\supset Q$, the cardinals of the fields $Q(b_{i})$ and $K$ are also $\vert I \vert$. In particular, if a field $K$ has transcendence degree over its prime field greater than a given infinite cardinal $b$, then $\vert K\vert>b$ too.
\end{lemma}
\begin{proof}
Since $\{x_{\a}\}_{\a \in \A}\subseteq Q(x_{\a})$ we have $\vert \A\vert\le \vert Q(x_{\a})\vert$. 
The polynomial ring in all the indeterminates $Q[x_{\a}]$ has cardinal $\vert \A\vert$. Furthermore, the field of rational functions $Q(x_{\a})$ has the same
cardinal $\vert \A\vert$. Note that if we have an extension of fields such as the smallest field is infinite and the largest one is algebraic over the first, then it is well known that both fields have the same cardinal. Consequently, since $K$ is algebraic over $Q(b_{i})$ (and this is infinite) we have
$\vert K\vert=\vert Q(b_{i})\vert=\vert I\vert$.
\end{proof}

\begin{remark}\rm\label{addon}
Consider a field $Q$ with $\vert Q \vert=\aleph_0$. If $\{b_1,\ldots,b_n\}$ is a finite transcendence basis of the extension $K\supset Q$, then the cardinal of $Q(b_1,\ldots,b_n)$ is $\aleph_0$.
\end{remark}

Next, we have a converse for the Lemma \ref{Leo}:

\begin{lemma}\label{Leo2}
Let $K$ be a field and $\{x_\a\}_{\a\in\A}$ an infinite set of indeterminates.
If $\vert K\vert>\vert\A\vert$, then the  transcendence degree of $K$ over its prime field is greater than $\vert\A\vert$.
\end{lemma}
\begin{proof} Assume $\vert K\vert>\vert\A\vert$ and $\A$ is not finite.
Let $\{b_i\}_{i\in I}$ be a transcendence basis of $K$ over $Q$ and $b$ the transcendence degree of $K$ over $Q$ (so $b=\vert\{b_i\}_{i \in I}\vert$).
Since $K$ is algebraic over $Q(b_i)$ we have 
\begin{equation}\label{Romeo}
\vert K\vert\le \aleph_0\vert Q(b_i)\vert
\end{equation}
This last assertion is true because we are able to construct an injective application between the set of equivalence classes of elements of $K$ with the same minimal polynomial and the set of polynomials in $Q(b_i)$.  
If $b$ is infinite
 we apply Lemma \ref{Leo}. From \eqref{Romeo} we obtain
$\vert K\vert\le b$. Since $\vert\A\vert<\vert K\vert$ we get the result.

Next, we prove that $b$ is necessarily infinite. If $b$ is finite, $\vert Q(b_i)\vert=\vert Q\vert=\aleph_0$ (Remark \ref{addon}) but $\vert K\vert>\aleph_0$ by hypothesis. 
Then equation \eqref{Romeo} gives $\aleph_0<\vert K\vert\le\aleph_0^2=\aleph_0$ which is a contradiction. 
\end{proof}

In view of the results of the previous Lemmas \ref{Leo} and \ref{Leo2} we can replace item {\it(ii)} of Theorem \ref{extension} for this another one:
\begin{equation}\label{Kaplansky}
\textit{(ii)'} \textit{ Let $\A$ have cardinal $a$. Let $b$ be the cardinal of the field $K$. Then $a < b$.}
\end{equation}
\section{The finite dimensional case}\label{fdc}

In this section $A$ will denote an algebra (associative or not) over a field $K$. Let $\Lambda$ be a finite indexing set. For a basis $\mathcal B= \{e_i\}_{i\in \Lambda}$ of $A$ (as a $K$-vector space), the \emph{structure constants} $\{\gamma_{ij}^k\}_{i, j, k \in \Lambda}$ relative to the basis $\mathcal B$ are those elements  $\gamma_{ij}^k\in K$ such that $e_ie_j=\sum_{k\in \Lambda}\gamma_{ij}^ke_k$ for every $i, j\in \Lambda$. For every pair $(i, j)\in \Lambda \times \Lambda$, denote the set $\Lambda_{ij}:= \{k\in \Lambda \ \vert \ \gamma_{ij}^k \neq 0\}$
and by $Q$ the prime field contained in $K$. Assume that the structure constants satisfy $\gamma_{ij}^k\in Q$. 
Define $\oplus_{i\in\Lambda} Qe_i$ and more generally for any associative, commutative an unital $Q$-algebra $R$ define $\oplus_{i\in\Lambda} Re_i$,
which is an $R$-algebra with the same structure constants (embedding $Q$ in $R$). Note that $\oplus_{i\in\Lambda} Re_i\cong (\oplus_{i\in\Lambda} Qe_i)\otimes_Q R$ as $R$-algebras. Furthermore, $\oplus_{i\in\Lambda} Ke_i=A$ and for any field extension $F\supset K$ we have $\oplus_{i\in\Lambda} Fe_i\cong A\otimes_K F = A_F$ as $F$-algebras. This is the usual scalar extension.

Note that the isomorphism question for finite-dimensional algebras can be translated into a problem of existence of zeros for suitable systems of algebraic equations. This is formulated in the next proposition.

\begin{proposition}\label{tecnicoUNO} 
Let $A$ and $B$ be finite dimensional 
$K$-algebras with bases $\mathcal B_A= \{e_i\}_{i\in \Lambda}$ and $\mathcal B_B=\{u_i\}_{i\in \Lambda}$, respectively. Let $\{\gamma_{ij}^k\}_{i, j ,k, \in \Lambda}$ and  $\{\tau_{ij}^k\}_{i, j, k \in \Lambda}$ be the structure constants of $A$ and $B$ relative to the bases $\mathcal B_A$ and $\mathcal B_B$, respectively. 

Then $A$ and $B$ are isomorphic $K$-algebras if and only if there exist 
$\{t_i^j\}_{i, j \in \Lambda}\cup \{s_i^j\}_{i, j \in \Lambda}\subseteq K$,  such that 
the following equations are satisfied:
\begin{equation}\label{ecuaciones1}
\begin{aligned}
\sum\limits_{k\in \Lambda}\gamma_{ij}^k t_k^n &  =   \sum\limits_{l, m\in \Lambda} \tau_{lm}^nt_i^lt_j^m
 & \text{for every } & i, j, n \in \Lambda,\\
 \sum\limits_{j\in \Lambda}t_i^js_j^k & =  \delta_{ik} &  \text{for every } &  i, k \in \Lambda.
\end{aligned}
\end{equation}
\end{proposition}
\begin{proof}
Assume first that $f\colon A\to B$ is an isomorphism of $K$-algebras.
Let $t_i^j\in K$ be such that $f(e_i)= \sum_{j\in \Lambda}t_i^j u_j$. The hypothesis of being $f$ a homomorphism of algebras means $f(e_ie_j)=f(e_i)f(e_j)$ for every $i, j \in \Lambda$. Then

$$
\begin{aligned}
f\left(\sum\limits_{k\in \Lambda}\gamma_{ij}^ke_k\right) & = \sum\limits_{l\in \Lambda}t_i^lu_l
 \sum\limits_{m\in \Lambda}t_j^mu_m, \\
 \sum\limits_{k\in \Lambda}\gamma_{ij}^kf(e_k) & = \sum\limits_{l, m\in \Lambda}t_i^lt_j^mu_lu_m,\\
 \sum\limits_{k\in \Lambda}\gamma_{ij}^k   \sum\limits_{n\in \Lambda}t_k^n u_n 
 & = \sum\limits_{l, m\in \Lambda,}t_i^lt_j^m\sum_{n\in \Lambda}\tau_{lm}^n u_n.
 \end{aligned}
 $$
 
 Then, taking into account the linear independence of the $u_n$'s we have, for any $i, j, n,$
 
 $$ 
  \sum\limits_{k\in \Lambda}\gamma_{ij}^k t_k^n = \sum\limits_{l, m\in \Lambda}t_i^lt_j^m \tau_{lm}^n.
  $$
 
 The following step is to impose that $f$ is bijective. If this is the case,  $f^{-1}$ denotes the inverse of $f$ and
 $f^{-1}(u_i)= \sum\limits_{j\in \Lambda} s_i^je_j$, where $s_i^j\in K$. Then 
 $$e_i=f^{-1}f(e_i) = f^{-1}\left(\sum\limits_{j\in \Lambda}t_i^j u_j\right)= \sum\limits_{j, k\in \Lambda}t_i^js_j^ke_k,$$
 hence
 $$ 
 \sum\limits_{j\in \Lambda}t_i^js_j^k = \delta_{ik}.
 $$
 
 The converse follows immediately. 
\end{proof}

Let $A$ and $B$ be $K$-algebras having the same finite dimension as $K$-vector spaces, and let  $\{\gamma_{ij}^k\}_{i, j, k, \in \Lambda}$ and  $\{\tau_{ij}^k\}_{i, j, k \in \Lambda}$  be the structure constants of $A$ and $B$, respectively, relative to certain bases indexed in a set  $\Lambda$.
Consider $X=\{T_i^j\}_{i, j \in \Lambda}\cup \{S_i^j\}_{i, j \in \Lambda}$ a set of indeterminates and let  $K[X]$ be the $K$-algebra of commuting polynomials in the indeterminates of $X$. For every $i, j, n \in \Lambda$, define
$$p_{ijn}:= \sum\limits_{k\in \Lambda}\gamma_{ij}^kT_k^n - \sum\limits_{l,m \in \Lambda } \tau_{lm}^nT_i^lT_j^m$$
and for every $i, k \in \Lambda$, 
$$q_{ik}:=\sum_{j \in \Lambda}T_i^jS_j^k-\delta_{ik}.$$

\begin{corollary}\label{correspondencia}
In the previous conditions, let $\mathfrak{i}_K$ be the ideal of $K[X]$ generated by the polynomials of the set $P_K:=\{p_{ijn}\}_{i, j, n \in \Lambda}\cup \{q_{ik}\}_{i, k \in \Lambda}$. Then the set of isomorphisms $f:A \to B$ is in one-to-one correspondence with the points of the algebraic variety $V(\mathfrak{i}_K)$, that is, the common zeros of the polynomials  in $P_K$.
\end{corollary}

Now that we have translated the isomorphism question of finite-dimensional algebras to a problem of varieties, we can exploit the Hilbert's Nullstellensatz Theorem in the following result.

\begin{theorem}\label{VueltaAlCole}
Let $K\subseteq F$ be a field extension, with $K$ algebraically closed. Consider $A$ and $B$ two $K$-algebras having finite dimension as $K$-vector spaces. Then $A_F$ and $B_F$ are isomorphic as $F$-algebras if and only if $A$ and $B$ are isomorphic as $K$-algebras.
\end{theorem}
\begin{proof}
Let $X=\{T_i^j\}_{i, j \in \Lambda}\cup \{S_i^j\}_{i, j \in \Lambda}$ be a set of commuting indeterminates, and consider the polynomial algebras $K[X]$ and $F[X]$. Take, as in Corollary \ref{correspondencia}, the ideals $\mathfrak{i}_K$ and $\mathfrak{i}_F$ of $K[X]$ and $F[X]$ generated by the polynomials in $P_K$ and $P_F$ respectively. The $K$-algebra monomorphism $i\colon K[X]\to F[X]$ maps $\mathfrak{i}_K$ into $\mathfrak{i}_F$ so that if $1\in \mathfrak{i}_K$, then $1\in \mathfrak{i}_F$. If $A_F$ and $B_F$ are isomorphic as $F$-algebras, then $V(\mathfrak{i}_F)\neq \emptyset$ by Corollary \ref{correspondencia}, which implies $1 \notin \mathfrak{i}_K$. Therefore $\mathfrak{i}_K\neq K[X]$ and by the Hilbert's Nullstellensatz Theorem we get $V(\mathfrak{i}_K)\neq \emptyset$. Again, by Corollary \ref{correspondencia}, this implies that $A$ and $B$ are isomorphic as $K$-algebras.
\end{proof}

\begin{remark}\rm
The hypothesis of being algebraically closed field cannot be eliminated. For example, consider the quaternion division algebra $\mathbb H$ and $M_2(\mathbb R)$, which are non isomorphic $\mathbb R$-algebras. However, their complexifications are isomorphic as $\mathbb C$-algebras.
\end{remark}
Firstly, as an immediate consequence of the previous theorem, we contemplate the case when we have two fields containing another algebraically closed.

\begin{corollary}\label{mismaCaracteristica} Let $F_1$ and $F_2$ be two fields containing an algebraically closed field $K$. Consider $A$ and $B$ be two $K$-algebras with finite dimension. Then $A_{F_1}$ and $B_{F_1}$ are isomorphic as $F_1$-algebras if and only if $A_{F_2}$ and $B_{F_2}$ are isomorphic as $F_2$-algebras.
\end{corollary}
\begin{proof} Let us consider $A_{F_1}$ and $B_{F_1}$ as $F_1$-algebras. By Theorem \ref{VueltaAlCole} these two algebras are isomorphic if and only if $A$ and $B$ are isomorphic as $K$-algebras. Again, by Theorem \ref{VueltaAlCole} $A$ and $B$ are isomorphic if and only if $A_{F_2}$ and $B_{F_2}$ are isomorphic as $F_2$-algebras. By transitivity, we have our claim.
\end{proof}

Secondly, we may have the situation when two algebraically closed fields contains another one which is not necessarily algebraically closed.

\begin{corollary} \label{dosalgcerrados}Let $F_1$ and $F_2$ be two algebraically closed fields containing a field $K$. Consider $A$ and $B$ two $K$-algebras with finite dimension. Then $A_{F_1}$ and $B_{F_1}$ are isomorphic as $F_1$-algebras if and only if $A_{F_2}$ and $B_{F_2}$ are isomorphic as $F_2$-algebras.
\end{corollary}
\begin{proof} First, we know that if $A$ is a $K$-algebra and $K_1,K_2$ are two fields verifying $K\subseteq K_1 \subseteq K_2$, then $(A_{K_1})_{K_2} \cong A_{K_2}$. Due to the fact that $F_1$ and $F_2$ are algebraically closed, it is clear that $\bar{K} \subseteq F_1,F_2$, with $\bar{K}$ denoting  the algebraic closure of $K$. Therefore, we can consider $A_{\bar{K}}$ and $B_{\bar{K}}$ two $\bar{K}$-algebras. We are now in the conditions of Corollary \ref{mismaCaracteristica} implying $(A_{\bar{K}})_{F_1} \cong (B_{\bar{K}})_{F_1}$ if and only if $(A_{\bar{K}})_{F_2} \cong (B_{\bar{K}})_{F_2}$.
\end{proof}

On the other hand, we can establish a more general case but taking into account that the algebraic closure of the prime field of the smallest one is contained in one of the largest fields.

\begin{corollary}\label{primefield} Let $K \subseteq F_i$ for $i=1,2$ field extensions. Consider $Q$ the prime field of $K$ and $A$ and $B$ two $K$-algebras of finite dimension whose structure constants (relative to suitable bases) are in $\bar{Q}$. Suppose $\bar{Q} \subseteq F_2$, if $A_{F_1} \cong B_{F_1}$, then $A_{F_2} \cong B_{F_2}$.
\end{corollary}

\begin{proof} Since $A_{F_1} \cong B_{F_1}$, then $A_{\bar{F_1}} \cong B_{\bar{F_1}}$ because $K \subseteq \bar{F_1},\bar{F_2}$. Applying Corollary \ref{dosalgcerrados} we have $A_{\bar F_2}\cong B_{\bar F_2}$.
Suppose $\{a_i\}$ and $\{b_j\}$ are respectively the bases of $A$ and $B$ as $K$-algebras with structure constants in $\bar{Q}$. Write $X= \oplus_{i} \bar{Q}a_i$ and $Y= \oplus_{j} \bar{Q}b_j$ as $\bar{Q}$-algebras. Therefore, we have that $X_{\bar{F_2}} \cong Y_{\bar{F_2}}$ because $A_{\bar{F_2}} \cong  B_{\bar{F_2}}$. Since $\bar{Q} \subseteq \bar{F_2}$, by Theorem \ref{VueltaAlCole}, $X \cong Y$. Finally, $A_{F_2} \cong B_{F_2}$ considering the fact that $\bar{Q} \subseteq F_2$.
\end{proof}

To end this section, we can restate all the results in terms of algebra functors which will give a more global vision. To begin with, we show that Theorem \ref{VueltaAlCole} could be rewritten in a functorial way giving Theorem \ref{zapaterobis}. In other words, both statements are equivalent.

\begin{theorem}\label{zapaterobis}
Let $K \subseteq F$ be a field extension with $K$ algebraically closed. Let $\F,\mathcal{G}\colon \alg_K \to \Alg_K$ be two extension invariant algebra functors and assume that $\F(K)$  and ${\mathcal G}(K)$ are finite-dimensional as $K$-algebras. Then $\F(K) \cong {\mathcal G}(K)$ if and only if $\F(F)\cong{\mathcal G(F)}$.
\end{theorem}

\begin{proof} Note that it is straightforward considering, in Theorem \ref{VueltaAlCole}, $A:=\F(K)$ and $B:=\mathcal{G}(K)$.
\end{proof}

Also, as a consequence of Theorem \ref{VueltaAlCole}, we get the following results in terms of extension invariant functors.

\begin{theorem}\label{zapatero}
Let $\F,\mathcal{G}\colon \rng\to \Alg_\Z$ be two extension invariant algebra functors and assume that for some field $K$  we have $\F(K)\cong{\mathcal G}(K)$ isomorphic as finite-dimensional K-algebras. Let $Q$ be the prime field of $K$ and $R$ any ring containing $\bar Q$ as a subring. Then $\F(R)\cong{\mathcal G(R)}$.
\end{theorem}
\begin{proof}First we take $A:=\F(\bar{Q})$ and $B:={\mathcal G}(\bar{Q})$. Since $\F$ is extension invariant, $\F(K)\cong{\mathcal G}(K)$ implies $\F({\bar K}) \cong{\mathcal G}({\bar K})$ as ${\bar K}$-algebras.  
Then, taking into account that $A$ and $B$ are finite-dimensional $\bar Q$-algebras,
by Theorem \ref{VueltaAlCole}, we know that $\F({\bar K})\cong{\mathcal G}({\bar K})$ implies $\F(\bar Q)\cong{\mathcal G}(\bar Q)$ as $\bar Q$-algebras. Finally $\F(R)\cong A\otimes_{\bar Q} R\cong B\otimes_{\bar Q} R \cong{\mathcal G}(R)$ as desired.
\end{proof}

Analogously to the equivalence between Theorem \ref{VueltaAlCole} and Theorem \ref{zapaterobis}, we can rewrite Corollaries \ref{mismaCaracteristica} and \ref{dosalgcerrados} in terms of functors. That is, under the corresponding hypotheses ($K$ algebraically closed field and $F_1$, $F_2$ containing $K$) in each result, we have that for $\F, \mathcal{G} \colon \alg_K \to \Alg_K$ extension invariant functors it is straightforward that $\F(F_1) \cong \mathcal{G}(F_1)$ if and only if $\F(F_2) \cong \mathcal{G}(F_2)$. So we have freedom to replace a field with another one and still to have an isomorphism, as long as both contain an algebraically closed one. Respectively for Corollary \ref{primefield} we could say $\F(F_1) \cong \mathcal{G}(F_1)$ implies $\F(F_2) \cong \mathcal{G}(F_2)$. 

We can specialize Theorem \ref{zapatero} to any extension invariant functor. Since our main motivation comes from Leavitt path algebras and related structures (path, Cohn and Steinberg algebras), we claim: 

\begin{corollary}\label{granuja}
Let $C_{K}^{X_{E_1}}(E_1)$ and $C_K^{X_{E_2}}(E_2)$ be isomorphic finite-dimensional relative Cohn path algebras. Let $Q$ be the prime field of $K$,
then for any field $K'$ of the same characteristic that $K$ such that $\bar{Q} \subseteq K'$, we have  $C_{K'}^{X_{E_1}}(E_1)\cong C_{K'}^{X_{E_2}}(E_2)$ as $K'$-algebras.
\end{corollary}

Also an immediate consequence of Corollary \ref{mismaCaracteristica} for the relative Cohn path algebras is the following: 

\begin{corollary}\label{granujilla} Let $K$ and $K'$ be algebraically closed fields of the same characteristic and assume that $C_K^{X_{E_1}}(E_1)$ is finite-dimensional. Then
$C_K^{X_{E_1}}(E_1)\cong C_K^{X_{E_2}}(E_2)$ if and only if $C_{K'}^{X_{E_1}}(E_1)\cong C_{K'}^{X_{E_2}}(E_2)$.
\end{corollary}
 
It is also true that if
$E_1,E_2$ are finite graphs with no cycles and $K$ algebraically closed, $K\subset F_i$ for $i=1,2$, then for the path algebras we have $F_1E_1\cong F_1E_2$ if and only if $F_2E_1 \cong F_2E_2$, and similarly for Steinberg, evolution and group algebras of finite dimension. For instance, for the finite-dimensional Steinberg algebras we have $A_{F_1}(G_1)\cong A_{F_1}(G_2)$ if and only if $A_{F_2}(G_1)\cong A_{F_2}(G_2)$ for groupoids $G_1$ and $G_2$. In the same vein, being $G_1$ and $G_2$ two finite groups
and using the group ring functors suitably, we have $F_1G_1\cong F_1G_2$ if and only if $F_2G_1\cong F_2G_2$. 

Being aware that, in the setting of relative Cohn path algebras, the finite-dimen\-siona\-lity imposes a strong restriction to the algebras under study, we will try to relax the hypothesis on finite-dimensionality of the algebras in the next section.

\section{The infinite dimensional case}\label{idc}
Now we will extend some of the results of the previous section to the case in which the algebras are of arbitrary dimension.
 Let $A$ and $B$ be $K$-algebras with bases $\mathcal B_A=\{e_i\}_{i\in \Lambda}$ and $\mathcal B_B=\{u_i\}_{i\in \Lambda}$ respectively. Let $f: A \to B$ be an isomorphism of $K$-algebras. Assume $f(e_i)= \sum_{j\in \Lambda}t_i^ju_j$ and $f^{-1}(u_i) =  \sum_{j\in \Lambda}s_i^je_j$ for certain $\{t_i^j\}_{i, j\in \Lambda}, \{s_i^j\}_{i, j\in \Lambda}\subseteq K$. By ${\rm Isom}_f(A, B)$ we will understand the set of all isomorphisms $g$ from $A$ into $B$ such that if $g(e_i)= \sum_{j\in \Lambda}{\overline t}_i^ju_j$ and $g^{-1}(u_i) =  \sum_{j\in \Lambda}{\overline s}_i^je_j$, then   $t_i^j=0$ implies ${\overline t}_i^j=0$ and $s_i^j=0$ implies ${\overline s}_i^j=0$. Note that, in particular, $f \in {\rm Isom}_f(A, B)$.
 
 As we did in Section \ref{fdc}, we can translate the isomorphism question for infinite-dimensional algebras in terms of the existence of zeros of a certain ideal of polynomials.

\begin{proposition}\label{correspondenciaInf} 
Let $A$ and $B$ be $K$-algebras with bases $\mathcal B_A=\{e_i\}_{i\in \Lambda}$ and $\mathcal B_B=\{u_i\}_{i\in \Lambda}$. Let $\{\gamma_{ij}^k\}_{i, j, k \in \Lambda}, \{\tau_{lm}^n\}_{l, m, n \in \Lambda}$ be the structure constants of $A$ and $B$ relative to the bases $\mathcal B_A$ and $\mathcal B_B$ respectively.

Assume that $f:A \to B$ is an isomorphism of $K$-algebras. Let $\{t_i^j\}_{i, j\in \Lambda}$ and $\{s_i^j\}_{i, j\in \Lambda}$ be subsets of $K$ such that $f(e_i)= \sum_{j\in \Lambda}t_i^ju_j$ and $f^{-1}(u_i) =  \sum_{j\in \Lambda}s_i^je_j$. Define
$$\Lambda_i^f:=\{j \in \Lambda \  \vert  \ t_i^j \neq 0\},\  \ \Lambda_i^{f^{-1}}  :=\{j \in \Lambda \ \vert \ s_i^j \neq 0\}, 
\ \ \Sigma_{i,j}  :=\{k\in\Lambda\ \vert\  \gamma_{ij}^k\ne 0\}.$$

Now, consider a set of indeterminates $X=\{T_i^j\}_{i,j\in\Lambda}\cup\{S_i^j\}_{i,j\in\Lambda}$ and the polynomials
\begin{equation}\label{ecuaciones2}
\begin{split}
   p_{ijn}& := \sum\limits_{k\in \Sigma_{i,j}}\gamma_{ij}^kT_k^n - \sum\limits_{l\in \Lambda_i^f, m\in\Lambda_j^f }\tau_{lm}^nT_i^lT_j^m, \ \ \  q_{ik} :=\sum_{j\in\Lambda_i^f}T_i^jS_j^k-\delta_{ik}, \\
    r_{ik} & :=\sum_{j\in\Lambda_i^{f^{-1}}}S_i^jT_j^k-\delta_{ik},  \ \ \ T_i^j \  (j\notin\Lambda_i^f), \; \; \; S_i^j \  (j\notin\Lambda_i^{f^{-1}}),
\end{split}
\end{equation}

\noindent
where $i,j,k,n\in\Lambda$.
Denote by $I_K^f$ the ideal of $K[X]$ generated by all the polynomials defined above.
Then there is a bijective map from $V(I_K^f)$ into ${\rm Isom}_f(A, B)$.
\end{proposition}

\begin{proof}
Note that, by construction, $\Lambda_i^f$,  $\Lambda_i^{f^{-1}}$ and $\Sigma_{i,j}$ are finite sets. In particular, this implies that the polynomials $p_{ijn}, q_{ik}$ and $r_{ik}$ are well-defined. 

Let  $z=((t_i^{\prime j}), (s_i^{\prime j}))$  be a zero in  $V(I_K^f)$; given any $i\in \Lambda$ we have
that the number of nonzero elements in $\{t_i^{\prime j}\}_{j\in \Lambda}$ and $\{s_i^{\prime j}\}_{j\in \Lambda}$ is finite. This allows us to
define $g: A \to B$ as the linear map having $(t_i^{\prime j})$ as the associated matrix relative to the bases 
$\mathcal B_A$ and $\mathcal B_B$. Let $h: B \to A$ be the linear map whose associated matrix relative to the bases $\mathcal B_B$ and $\mathcal B_A$ is  $(s_i^{\prime j})$.  Since $z$ is a zero of the polynomials in (\ref{ecuaciones2}), it is clear that $g$ is a homomorphism of $K$-algebras and that $hg=1_A$ and $gh=1_B$ (similarly to Proposition \ref{tecnicoUNO}). Moreover, taking into account the construction of $g$ we deduce that  $g\in {\rm Isom}_f(A, B)$. This way, we can define a map $\varphi: V(I_K^f) \to {\rm Isom}_f(A, B)$ such that $z\mapsto g$.

Conversely, take $g\in {\rm Isom}_f(A, B)$. Denote by $(a_i^j)$ and by $(b_i^j)$ the matrices of $g$ and $g^{-1}$ relative to the bases $\mathcal B_A$ and $\mathcal B_B$. 
Notice that \begin{equation}\label{TyS}\begin{cases}a_i^j=0 & \text{ if } j\notin\Lambda_i^f\cr b_i^j=0 & \text{ if } j\notin \Lambda_i^{f^{-1}}.\end{cases}
\end{equation}
Since $g$ is the inverse of  $g^{-1}$ (and conversely), we have that $z=((a_i^j), (b_i^j))$ is a zero of $\{q_{ik}\}_{i, k \in \Lambda}\cup \{r_{ik}\}_{i, k \in \Lambda}$. Apply that $g$ is a homomorphism of $K$-algebras to get that $z$ is a zero of $\{p_{ijn}\}_{i, j, n \in \Lambda}$ and also of $T_i^j$ and $S_i^j$ because of (\ref{TyS}). Thus we can define a map $\psi\colon {\rm Isom}_f(A, B) \to V(I_K^f)$ and 
it is not difficult to prove that $\psi$ is the inverse of 
$\varphi$, again similarly to Proposition \ref{tecnicoUNO}.
\end{proof}
Since the algebras we are dealing with are possibly infinite dimensional, in order to apply Theorem \ref{extension}, we need to consider additional hypothesis about the cardinal of the ground field to proceed similarly as in Section \ref{fdc}.

\begin{theorem}\label{VueltaAlCole2}
Let $K\subseteq F$ be a field extension with $K$ algebraically closed. Let $A$ and $B$ be $K$-algebras of the same dimension $\omega$. Assume that the cardinal of $K$ is strictly higher than $\omega$. Then $A_F$ and $B_F$ are isomorphic as $F$-algebras if and only if $A$ and $B$ are isomorphic as $K$-algebras.
\end{theorem}
\begin{proof} 
For the nontrivial part assume that $A_F\cong B_F$. Let 
$f\colon A_F\to B_F$ be an isomorphism. Let $\Lambda$ be a set of cardinal $\omega $.
Fix bases $\B_A=\{e_i\}_{i\in\Lambda}$ of $A$ and
$\B_B=\{u_i\}_{i\in\Lambda}$ of $B$.
Assume that the structure constants of $A$ and $B$ relative to the bases $\B_A$ and $\B_B$ are respectively given by:
$$e_ie_j=\sum_{k \in \Lambda} \gamma_{ij}^k e_k,\qquad 
u_iu_j=\sum_{k\in \Lambda} \tau_{ij}^k u_k.$$
Note that the structure constants of $A_F$ and $B_F$ relative to the bases $\{u_i \otimes 1\}_{i\in \Lambda}$ and $\{e_i \otimes 1\}_{i \in \Lambda}$ are in $K$. 
Assume that $f(e_i \otimes 1)=\sum_{j\in \Lambda} t_i^j(u_j \otimes 1)$ and 
$f^{-1}(u_i \otimes 1)=\sum_{j\in \Lambda} s_i^j (e_j \otimes 1)$ for any $i\in\Lambda$.
Define now the sets $\Lambda_i^f$, $\Lambda_i^{f^{-1}}$ and $\Sigma_{i,j}$ as in Proposition \ref{correspondenciaInf}.
Consider the set of
indeterminates $X=\{T_i^j\}_{i,j\in\Lambda}\cup\{S_i^j\}_{i,j\in\Lambda}$. In the polynomial algebras $K[X]$ and $F[X]$, define respectively the ideals $I_K^f$ and $I_F^f$ generated by the polynomials 
$p_{ijn}$, $q_{ik}$, $r_{ik}$,  $T_i^j$ (with $j\notin\Lambda_i^f)$ and $S_i^j$ (with $j\notin\Lambda_i^{f^{-1}}$). Then $V(I_F^f)\ne\emptyset$ whence $1\notin I_F^f$. Observe that $I_K^f \subseteq I_F^f$ because the polynomials defined as in \eqref{ecuaciones2} have their coefficients in $K$. Thus $1\notin I_K^f$ implying 
$V(I_K^f)\ne\emptyset$  by Hilbert’s Nullstellensatz Theorem \ref{extension} taking into account condition (\ref{Kaplansky}). So there is an isomorphism $A\to B$.
\end{proof}

\begin{corollary}\label{mismaCaracteristica2} Let $F_1$ and $F_2$ be two fields containing an algebraically closed field $K$. Consider $A$ and $B$ two $K$-algebras with the same dimension $\omega$. Suppose that the cardinal of $K$ is strictly higher than $\omega$. Then $A_{F_1}$ and $B_{F_1}$ are isomorphic as $F_1$-algebras if and only if $A_{F_2}$ and $B_{F_2}$ are isomorphic as $F_2$-algebras.
\end{corollary}
\begin{proof} Let us consider $A_{F_1}$ and $B_{F_1}$ as $F_1$-algebras. By Theorem \ref{VueltaAlCole2} these two algebras are isomorphic if and only if $A$ and $B$ are isomorphic as $K$-algebras. Again, by Theorem \ref{VueltaAlCole2}, we have that $A$ and $B$ are isomorphic if and only if $A_{F_2}$ and $B_{F_2}$ are isomorphic as $F_2$-algebras. By transitivity, we have our claim.
\end{proof}

\begin{corollary} \label{dosalgcerrados2}Let $F_1$ and $F_2$ be two algebraically closed fields containing a field $K$. Suppose that $A$ and $B$ are two $K$-algebras with the same dimension $\omega$ and the cardinal of $K$ is strictly higher than $\omega$.  Then $A_{F_1}$ and $B_{F_1}$ are isomorphic as $F_1$-algebras if and only if $A_{F_2}$ and $B_{F_2}$ are isomorphic as $F_2$-algebras.
\end{corollary}
\begin{proof}
Due to the fact that $F_1$ and $F_2$ are algebraically closed, it is clear that $\bar{K} \subseteq F_1,F_2$. Therefore, we can consider $A_{\bar{K}}$ and $B_{\bar{K}}$ as two $\bar{K}$-algebras. Observe that the cardinal of $\bar{K}$ is bigger than $\omega$. We are now in the conditions of Corollary \ref{mismaCaracteristica2} implying $(A_{\bar{K}})_{F_1} \cong (B_{\bar{K}})_{F_1}$ if and only if $(A_{\bar{K}})_{F_2} \cong (B_{\bar{K}})_{F_2}$.
\end{proof}
\begin{corollary}\label{primefield2} Let $K \subseteq F_i$ for $i=1,2$ be extension of fields. Consider $A$ and $B$ two $K$-algebras with the same dimension $\omega$. Suppose that the cardinal of $K$ is strictly higher than $\omega$ and $\bar{K} \subseteq F_2$. If $A_{F_1} \cong B_{F_1}$ then $A_{F_2} \cong B_{F_2}$.
\end{corollary}
\begin{proof}  Since $A_{F_1} \cong B_{F_1}$, then $A_{\bar{F_1}} \cong B_{\bar{F_1}}$ taking into account $K \subseteq \bar{F_1}$. By Corollary \ref{mismaCaracteristica2} we have $A_{\bar{K}} \cong B_{\bar{K}}$. Because of $\bar{K} \subseteq F_2$, we get $A_{F_2} \cong B_{F_2}$.
\end{proof}

Finally, we can express the previous results in terms of extension invariant algebra functors. Concretely, Theorem \ref{VueltaAlCole2} gives Theorem \ref{VueltaAlCole2Funtorial} and from Corollary \ref{mismaCaracteristica2} we obtain Theorem \ref{aporotracosa}.

\begin{theorem}\label{VueltaAlCole2Funtorial}
Let $K \subseteq F$ be a field extension with $K$ algebraically closed. Let $\F,\mathcal{G}\colon \alg_K \to \Alg_K$ be two extension invariant algebra functors and assume that 
$\vert K\vert>\dim(\F(K)),\dim(\mathcal{G}(K))$.
Then $\F(K) \cong {\mathcal G}(K)$ if and only if $\F(F)\cong{\mathcal G(F)}$.
\end{theorem}

\begin{theorem}\label{aporotracosa} Let $F_1$ and $F_2$ be two fields containing an algebraically closed field $K$. Consider $\F$ and $\mathcal{G}$ two extension invariant algebra functors with  $\vert K\vert>\dim(\F(K)),\dim(\mathcal{G}(K))$. Then $\F(F_1)$ and $\mathcal{G}(F_1)$ are isomorphic as $F_1$-algebras if and only if $\F(F_2)$ and $\mathcal{G}(F_2)$ are isomorphic as $F_2$-algebras.
\end{theorem}

In fact, if $A$ and $B$ are two $K$-algebras such that $\dim(A)=\dim(B)=\aleph_0$,
then for any uncountable algebraically closed field $K$ we meet the hypothesis of Corollary \ref{mismaCaracteristica2}. 
For the path algebras, if $E_1$ and $E_2$ are countable graphs and $K\subset F_i$ $i=1,2$, we have that $F_1E_1\cong F_1E_2$ if and only if $F_2E_1 \cong F_2E_2$. Moreover, for Steinberg algebras of countable dimension it satisfies $A_{F_1}(G_1)\cong A_{F_1}(G_2)$ if and only if $A_{F_2}(G_1)\cong A_{F_2}(G_2)$ for groupoids $G_1$ and $G_2$. In order to highlight the corresponding result, in the case of Leavitt path algebras, we establish the following corollary:

\begin{corollary}\label{andale} Suppose that $E_1$ and $E_2$ are countable graphs, $K$ an uncountable algebraically closed field and $K \subseteq F_1, F_2$ two field extensions. Then $L_{F_1}(E_1)\cong L_{F_1}(E_2)$ if and only if $L_{F_2}(E_1)\cong L_{F_2}(E_2)$.
\end{corollary}

Note that if $K$ is an uncountable algebraically closed field, $E_1$ and $E_2$ countable graphs and $R, S$  commutative and unital $K$-algebras, then $L_R(E_1)\cong L_R(E_2)$ as $R$-algebras if and only if $L_S(E_1)\cong L_S(E_2)$ as $S$-algebras. Indeed, assuming $L_R(E_1)\cong L_R(E_2)$ we consider the field $K'=R/\mathfrak{m}$ where $\mathfrak{m}$ is any maximal ideal of $R$.  We have a field extension $K\subset K'$ and from the isomorphism $L_R(E_1)\cong L_R(E_2)$ we deduce an isomorphism $L_{K'}(E_1)\cong L_{K'}(E_2)$ so that Corollary \ref{andale} implies the existence of an isomorphism $L_K(E_1)\cong L_K(E_2)$ whence $L_S(E_1)\cong L_S(E_2)$.

\end{document}